\documentclass[11pt]{amsart}
\usepackage{hyperref}
\usepackage{amsxtra}
\usepackage{amssymb}
\usepackage{comment}
\usepackage{mathrsfs} 
\usepackage{enumerate}
\usepackage{array}
\usepackage{xcolor}
\usepackage{cancel}
\usepackage{tikz}
\theoremstyle{plain}

\newtheorem{theorem}{Theorem}[section]
\newtheorem{lemma}[theorem]{Lemma}
\newtheorem{definition-theorem}[theorem]{Definition-Theorem}
\newtheorem{proposition}[theorem]{Proposition}
\newtheorem{corollary}[theorem]{Corollary}
\newtheorem{definition}[theorem]{Definition}

\theoremstyle{definition}
\newtheorem{example}[theorem]{Example}

\newcommand{\bbr}{\mathbb{R}}

\newcommand{\bbc}{\mathbb{C}}

\newcommand{\bbn}{\mathbb{N}}

\newcommand{\wt}{\widetilde}

\newcommand{\vocab}[1]{\textit{#1}}

\DeclareMathOperator{\vspan}{\mathrm{Span}}

\DeclareMathOperator{\Ai}{{\mathrm{Ai}}}
\DeclareMathOperator{\Bi}{\mathrm{Bi}}
\DeclareMathOperator{\ord}{\mathrm{ord}}
\DeclareMathOperator{\cord}{\mathrm{cord}}
\DeclareMathOperator{\DGal}{\mathrm{DGal}}
\DeclareMathOperator{\sym}{\mathrm{sym}}

\usepackage{amsmath,calligra,mathrsfs}
\DeclareMathOperator{\shom}{\mathscr{H}\text{\kern -3pt {\calligra\large om}}\,}

\begin{document}
\title[Algebras of commuting differential operators
for kernels of Airy type]
{Algebras of commuting differential operators \\
for integral kernels of Airy type}
\dedicatory{To the memory of Harold Widom, with admiration}
\author[W. Riley Casper]{W. Riley Casper}
\address{
Dept of Mathematics \\
California State University Fullerton\\
Fullerton, CA 92831\\
U.S.A.
}
\email{wcasper@fullerton.edu}

\author[F. Alberto Gr\"unbaum]{F. Alberto Gr\"unbaum}
\address{
Dept of Mathematics \\
University of California, Berkeley \\
Berkeley, CA 94720 \\
U.S.A.
}
\email{grunbaum@math.berkeley.edu}

\author[Milen Yakimov]{Milen Yakimov}
\address{
Dept of Mathematics \\
Northeastern University \\
Boston, MA 02115 \\
U.S.A.
}
\email{m.yakimov@northeastern.edu}

\author[Ignacio Zurri\'an]{Ignacio Zurri\'an}
\address{
Dept of Mathematics \\
Universidad Nacional de C\'ordoba, C\'ordoba, X5016HUA, Argentina
}
\email{ignacio.zurrian@fulbrightmail.org}
\begin{abstract}
Differential operators commuting with integral operators were discovered in the work of C. Tracy and H. Widom  \cite{TW1,TW2} and used to derive asymptotic expansions of the Fredholm determinants of integral operators arising in random matrix theory. Very recently, it has been proved that all rational, symmetric Darboux transformations of the Bessel, Airy, and exponential bispectral functions give rise to commuting integral and differential operators \cite{CGYZ1,CGYZ2,CY}, vastly generalizing the known examples in the literature. In this paper, we give a classification of the the rational symmetric Darboux transformations of the Airy function in terms of the fixed point submanifold of a differential Galois group acting on the Lagrangian locus of the (infinite dimensional) Airy Adelic Grassmannian and initiate the study of the full algebra of differential operators commuting with each of the integral operators in question. We leverage the general theory of \cite{CY} to obtain explicit formulas 
for the two differential operators of lowest orders that commute with each of
the level one and two integral operators obtained in the Darboux process.
Moreover, we prove that each pair of differential operators commute with each other.
The commuting operators in the level one case are shown to satisfy an algebraic relation defining an elliptic curve.
\end{abstract}
\maketitle
\section{Introduction}
Our contribution to this volume bears a connection with a phenomenon uncovered by Craig Tracy and Harold Widom \cite{TW2} in their work on level spacing in Random Matrix Theory.
For a double scaling limit at the ``edge of the spectrum'' they observed that the resulting integral operator with the Airy kernel acting on an appropriate interval admits a 
commuting second order differential operator. This highly exceptional fact is put to good use in section IV of their paper where a number of asymptotic results 
for several quantities of interest are given.

In the context of Random Matrix Theory the existence of such a commuting pair of operators had been exploited earlier, for instance in work by M. Mehta \cite{Mehta}
and W. Fuchs \cite{Fuchs}. In this case one is interested in the ``bulk of the spectrum'' and the role of the Airy kernel is taken up by the more familiar sinc kernel. 
Both of these situations deal with the Gaussian Unitary Ensemble.

The consideration of either the Laguerre or the Jacobi  ensembles at the ``edge of the spectrum'' gives rise to the Bessel kernel.
This case, as well as the corresponding commuting pair of integral-differential operators is considered by C. Tracy and H. Widom in \cite{TW1}.
There, once again, this exceptional fact is exploited in section III to derive a number of important asymptotic results.

In this paper we concentrate on the ``exceptional fact'' mentioned above in three different situations relevant to Random Matrix theory. This fact
had emerged in other areas of mathematics. In a ground-breaking collection of papers by D. Slepian, H. Landau and H. Pollak done at Bell labs in the 1960's
\cite{LP1,LP2,S,SV,S1,S2,SP} instances of this phenomenon were discovered and used in a key way in communication-signal processing theory.
In fact, some precedents can be traced further back, see \cite{bateman,Ince}. 
For an up-to-date treatment of the numerical issues involving the prolate spheroidal function, one can see \cite{ORX}. 

Incidentally in the case of the Bessel kernel the existence of a commuting operator was already proved by D. Slepian, while the situation of the Airy kernel appears for
the first time in C. Tracy and H. Widom's paper mentioned above.  The so called ``prolate spheroidal wave functions,'' which arise in the case of the sinc kernel and 
their corresponding integral-differential pair of operators, have played an important role in areas far removed from signal processing that motivated
the research of Slepian and collaborators.
We give only two instances of this, but we are sure that other people can provide other examples: the paper by J. Kiukas and R. Werner \cite{KW} 
in connection with Bell's inequalities,
and the program by A. Connes in connection with the Riemann hypothesis with C. Consani, M. Marcolli and H. Moscovici \cite{CC,CM,CMo}. 
 
One should also mention that the Airy function itself and variants of it have played an important role in other very active areas of current research, 
such as quantum gravity and intersection theory on moduli space of curves, see \cite{Kon,Wit}. 

In all the three instances discussed above (namely the sinc, Bessel and Airy kernels), the commuting differential operator has been found by a direct computation that relies heavily on integration by parts.
The interest in understanding and extending this exceptional phenomenon in a variety of other situations has produced some few more examples, see
\cite{BG,CG,G0,G1,G2,GLP,GPZ2}.

The bispectral problem formulated 1986 in \cite{DG} aimed at a conceptual understanding of the phenomenon of integral operators 
admitting a commuting differential operator. The idea is that all known kernels with this property are built from bispectral functions, that 
is functions in two complex variables that are eigenfunctions of differential operators in each of them. There has been a substantial 
amount of research on this problem \cite{GH,Iliev}, which started with the classification of all bispectral differential operators of second order \cite{DG} and 
culminated in the classification of bipectral functions of rank 1 in \cite{Wilson1} 
and the construction of bispectral functions of arbitrary rank via Darboux transformations \cite{BHY2,KR} and automorphisms of the first Weyl algebra \cite{BHY1,BHY3}. 

Since the mid 80s, the belief that the two problems, bispectrality and the existence of a commuting pair made up of a differential and an integral operator were closely
connected has been driving research on both fronts. However, for a long time there no general argument proving that bispectral functions give rise to kernels 
of integral operators with the commutativity property. This was finally settled in \cite{CY} where it was proved to be the case 
for self-adjoint bispectral functions of rank 1 and 2.

More recently we proved in \cite{CGYZ1,CGYZ2} that all bispectral functions of rank 1 give rise to integral operators that reflect a differential operator
rather than plain commute with it.  

All of the previous results on integral operators address the construction of a single differential operator commuting with it. 
The purpose of this paper is to initiate the systematic study of the \emph{algebras} of differential operators that commute with a given integral operator.
We start with the Airy example considered by C. Tracy and H. Widom and consider all self-adjoint bispectral Darboux transformations. This is an infinite 
dimensional manifold which sits canonically in the infinite dimensional Grassmannian of all Darboux transformation from the Airy function, obtained from 
factorizations of polynomials of the Airy operator 
\begin{equation}
\label{Airy-oper}
L(x, \partial_x) = \partial^2 -x.
\end{equation}
We give a conceptual classification of the former manifold as the fixed point set of a Lagrangian Grassmannian with respect to the canonical
action of the associated differential Galois group. The Lagrangian Grassmannian in question is the sub-Grassmannian with 
respect to a canonical symplectic form. We consider the first two instances of self-adjoint bispectral Darboux transformations coming 
from factorizations of 
\[
(L-t_1)^2 \quad \mbox{and} \quad (L- t_2)^4
\]
of the form $P^* P$ for a differential operator $P(x, \partial_x)$ with rational coefficients. The corresponding bispectral functions, referred to here as \vocab{level one} and \vocab{level two} bispectral functions, are significantly more complicated than the bispectral Airy function $\mathrm{Ai}(x+z)$.
The integral operators that they give rise to depend on parameters classifying different factorizations.
For each integral operator, we compute explicitly the differential operators of the lowest two orders and prove that they are algebraically dependent.
In the level one situation, the commuting operators have order $4$ and $6$.  They generate the algebra of all differential operators commuting with the integral operator and satisfy an algebraic relation which happens to be an elliptic curve.  In the level two situation, the lowest two commuting operators have order $10$ and $12$. However, we are also able to find commuting operators of order $14$, $16$, and $18$ and to prove that these differential operators commute with each other.  In a future publication, we will return to the problem of studying algebras of differential operators commuting with a fixed integral operator and will present general structural results for the algebra of differential operator commuting with all integral operators which are built from bispectral functions, and which are motivated by the examples in this paper.

This paper is written as a small token of admiration and gratitude to the amazing mathematical work of Harold Widom.
Widom started mathematical life as an algebraist working with Irving Kaplansky at Chicago, before becoming mainly an analyst through the influence of Mark Kac at Cornell.
This paper uses tools from both analysis and algebra, uniting Widom's dual mathematical history.
His influence will be a lasting one, and we will miss him badly.
\section{Bispectral functions, Fourier algebras and prolate spheroidal type commutativity}
\label{sec2}
\subsection{Bispectrality and Fourier Algebras}
For an open subset  $U \subseteq \bbc$, denote by $\mathfrak D(U)$ the algebra of differential operators on $U$ with meromorphic coefficients.
\begin{definition} \cite{DG}
Let $U$ and $V$ be two domains in $\bbc$. A nonconstant meromorphic function $\Psi(x,z)$ defined on $U\times V \subseteq \bbc^2$
is called {\em{bispectral}} if there exist differential operators 
$B(x, \partial_x) \in\mathfrak D(U)$ and $D(z, \partial_z) \in\mathfrak D(V)$ such that
\begin{align*}
B(x, \partial_x) \Psi(x,z) &= g(z)\Psi(x,z)
\\
D(z, \partial_z) \Psi(x,z) &= f(x)\Psi(x,z)
\end{align*}
for some nonconstant functions $f(x)$ and $g(z)$ meromorphic on $U$ and $V$, respectively.
\end{definition}
Denote by $\Ai(x)$ the classical Airy function. The function
\[
\Psi_{\Ai}(x,z) := \Ai (x+ z)
\]
 is bispectral because
\begin{equation}
\label{Airyu-bisp}
L(x, \partial_x) \Psi_{\Ai}(x,z)  = z \Psi_{\Ai}(x,z) \quad \mbox{and} \quad
L(z, \partial_z) \Psi_{\Ai}(x,z)  = x \Psi_{\Ai}(x,z), 
\end{equation}
where $L(x, \partial_x)$ is the Airy operator \eqref{Airy-oper}. 
The differential equations satisfied by a bispectral function are captured by the following definition.
\begin{definition} \cite{BHY1}
Let $\Psi(x,z)$ be a bispectral meromorphic function defined on $U\times V \subseteq \bbc^2$. 
Define the {\em{left}} and {\em{right Fourier algebras}} of differential operators for $\Psi$ by
\begin{align*}
{\mathfrak{F}}_x(\Psi) = \{R(x, \partial_x) \in \mathfrak D(U) : \, &\text{there exists a differential operator $S(z, \partial_z) \in \mathfrak D(V)$}  \\
&\text{satisfying $R(x, \partial_x)\Psi(x,z) =  S(z, \partial_z) \Psi(x,z)$}\}
\end{align*}
and 
\begin{align*}
{\mathfrak{F}}_z(\Psi) = \{S(z, \partial_z) \in \mathfrak D(V) : \, &\text{there exists a differential operator $R(x, \partial_x) \in \mathfrak D(U)$}  \\
&\text{satisfying $R(x, \partial_x) \Psi(x,z) = S(z, \partial_z) \Psi(x,z)$}\}.
\end{align*}
\end{definition}

By \cite[Proposition 2.4]{CY}, for every bispectral meromorphic function $\Psi : U \times V \to \bbc$, 
there exists a canonical anti-isomorphism 
\[
b_\Psi : {\mathfrak{F}}_x(\Psi)  \to {\mathfrak{F}}_z(\Psi), 
\]
given by $b_\Psi(R(x, \partial_x)) = S(z, \partial_z)$ if 
\[
R(x, \partial_x)\Psi(x,z) =  S(z, \partial_z) \Psi(x,z).
\]
We call this the {\em{generalized Fourier map}} associated to $\Psi(x,z)$. 
Define the {\em{co-order}} of an element $R(x, \partial_x) \in  {\mathfrak{F}}_x(\Psi)$ by
\[
\cord R := \ord b_\Psi(R). 
\] 
Analogously, we define the co-order of $S(z, \partial_z) \in  {\mathfrak{F}}_z(\Psi)$ by $\cord S := \ord b_\Psi^{-1}(S)$.
The Fourier algebras of $\Psi(x,z)$ have natural $\bbn \times \bbn$-filtrations:
\begin{align*}
&{\mathfrak{F}}_x(\Psi)^{\ell, m} = \{ R(x, \partial_x) \in {\mathfrak{F}}_x(\Psi) : \, \ord R \leq \ell, \cord R \leq m \},  \\
&{\mathfrak{F}}_z(\Psi)^{m, \ell} = \{ S(z, \partial_z) \in {\mathfrak{F}}_z(\Psi) : \, \ord S \leq m, \cord S \leq \ell \}, 
\end{align*}
where $\bbn = \{0,1,\ldots\}$ and $b_\Psi({\mathfrak{F}}_x(\Psi)^{\ell, m}) = {\mathfrak{F}}_z(\Psi)^{m, \ell}$.
The commutative algebras
\[
{\mathfrak{B}}_x(\Psi) := \bigcup_{\ell \geq 0} {\mathfrak{F}}_x(\Psi)^{\ell, 0} \quad \mbox{and} \quad 
{\mathfrak{B}}_z(\Psi) := \bigcup_{m \geq 0} {\mathfrak{F}}_z(\Psi)^{0, m} 
\]
are precisely the algebras of differential operators in $x$ and $z$, respectively, for which $\Psi(x,z)$ 
is a eigenfunction. 

\begin{example} 
\label{ex:Airy}
The Airy bispectral function $\Psi_{\Ai}(x,z)$ satisfies 
\begin{align*}
L(x, \partial_x) \Psi_{\Ai}(x,z) &= z \Psi_{\Ai}(x,z),  \\
\partial_x \Psi_{\Ai}(x,z) &= \partial_z \Psi_{\Ai}(x,z), \\
x \Psi_{\Ai}(x,z) &= L(z, \partial_z) \Psi_{\Ai}(x,z).
\end{align*}
The Fourier algebras ${\mathfrak{F}}_x(\Psi_{\Ai})$ and ${\mathfrak{F}}_z(\Psi_{\Ai})$ coincide with the first Weyl algebra in the variables $x$ and $z$, respectively, 
and the generalized Fourier map $b_{\Psi_{\Ai}}$ is the anti-isomorphism from the first Weyl algebra in $x$ to the first Weyl algebra in $z$ given by
\[
b_{\Psi_{\Ai}} (x) = \partial_z^2 -z, \; \; b_{\Psi_{\Ai}} (\partial_x) = \partial_z.
\]
Furthermore,
\[
\dim {\mathfrak{F}}_x(\Psi_{\Ai})^{2 \ell, 2 m} = \ell m + \ell + m + 1,
\]
see \cite[Sect. 3.1 and Lemma 5.5]{CY}.
On the level of Wilson's adelic grassmannian, the anti-isomorphism $b_\psi$ is equivalent to Wilson's bispectral involution \cite{Wilson1}.
More generally, every anti-automorphism of the first Weyl algebra determines a bispectral function
as proved in \cite{BHY3}.
\end{example}
\begin{definition} 
\label{def:ratDarbAi}
A rational Darboux transformation from the bispectral Airy function $\Psi_{\Ai}(x,z)$ is a function of the form 
\begin{equation}
\label{bispDarb1}
\Psi(x,z) := \frac{P(x, \partial_x) \Psi_{\Ai}(x,z)}{q(z) p(x)} 
\end{equation}
such that 
\begin{equation}
\label{bispDarb2}
\Psi_{\Ai}(x,z) = Q(x, \partial_x) \frac{\Psi(x,z)}{\widetilde q(z) \widetilde p(x)} 
\end{equation}
for some differential operators $P$ and $Q$ with polynomial coefficients and polynomials $p(x)$, $\widetilde p(x)$, 
$q(z)$ and $\widetilde q (z)$  with coefficients in $\bbc$.
We define the {\em{bidegree}} of such a transformation to be the pair $(\ord P, \cord P)$.   
\end{definition}
In this setting we have $Q, P \in {\mathfrak{F}}_x(\Psi_{\Ai})$, $\widetilde p(x), p(x) \in {\mathfrak{F}}_x(\Psi_{\Ai})^{0, m}$ and 
$\widetilde q(x), q(x) \in {\mathfrak{F}}_z(\Psi_{\Ai})^{0,\ell}$ for some $\ell, m \in \bbn$. Furthermore,  
eqs. \eqref{bispDarb1}--\eqref{bispDarb2} imply that 
\[
Q(x, \partial_x) \frac{1}{\widetilde p(x) p(x)} P(x, \partial_x) \Psi_{\Ai}(x,z) = {\widetilde q(z) q(z)} \Psi_{\Ai}(x,z),
\]
and thus by Example \ref{ex:Airy}, 
\begin{equation}
\label{QP}
Q(x, \partial_x) \frac{1}{\widetilde p(x) p(x)} P(x, \partial_x) = \widetilde q(L(x, \partial_x)) q(L(x, \partial_x)).
\end{equation}
\begin{theorem} \cite{BHY1,BHY3,KR} All rational Darboux transformations of the bispectral Airy function $\Psi_{\Ai}(x,z)$ are bispectral functions. More precisely, 
if $\Psi(x,z)$ is as in Definition \ref{def:ratDarbAi}, then it satisfies the spectral equations
\begin{align*}
\frac{1}{p(x)} P(x, \partial_x) Q(x, \partial_x)  \frac{1}{\wt{p}(x)} \Psi(x,z) &= q(z) \wt{q}(z) \Psi(x,z), \\
\frac{1}{q(z)} b_{\Psi_{\Ai}}(P)(z, \partial_z) b_{\Psi_{\Ai}}(S)(z, \partial_z)  \frac{1}{\wt{q}(z)} \Psi(x,z) & = p(x) \wt{p}(x) \Psi(x,z).
\end{align*} 
\end{theorem}
\subsection{Prolate Spheroidal Type Commutativity}
A rational Darboux transformation $\Psi(x,z)$ of the bispectral Airy function of bidegree $(d_1, d_2)$ 
is called {\em{self-adjoint}} if it has a presentation as in Definition \ref{def:ratDarbAi} 
such that 
\[
Q(x, \partial_x) = P^*(x, \partial_x)
\]
and $\widetilde p(x) = p(x)$, $\widetilde q(z) = q(z)$. Here $P \mapsto P^*$ denotes the formal adjoint.
It follows from \eqref{QP} that $P$ has even order.
A rational Darboux transformation $\Psi(x,z)$ of the Airy bispectral function $\Psi_{\Ai}(x,z)$ is self-adjoint if and only if 
the spectral algebras ${\mathfrak{B}}_x(\Psi)$ and ${\mathfrak{B}}_z(\Psi)$ are preserved under the formal adjoint, 
and this condition is satisfied if and only if $\Psi(x,z)$ is an eigenfunction of nonconstant, formally symmetric differential operators
in $x$ and $z$ (i.e., operators that are fixed by the formal adjoint), see \cite[Remark 3.17 and Proposition 3.18]{CY}.

For self-adjoint rational Darboux transformations $\Psi(x, z)$ of $\Psi_{\Ai}(x,z)$, both Fourier algebras 
${\mathfrak{F}}_x(\Psi)$ and ${\mathfrak{F}}_z(\Psi)$ are preserved under the formal adjoint and 
\begin{equation}
\label{ba}
(b_\Psi (R))^* = b_\Psi(R^*) \quad \mbox{for all} \quad R \in {\mathfrak{F}}_x(\Psi), 
\end{equation}
see \cite[Proposition 3.24 and 3.25]{CY}. Define
\[
{\mathfrak{F}}_{x, \sym}(\Psi) := \{ R \in {\mathfrak{F}}_{x}(\Psi) : \, R^* = R \}.
\]
By \eqref{ba} for all $R \in {\mathfrak{F}}_{x, \sym}(\Psi)$, 
\[
(b_{\Psi} (R))^* = b_{\Psi} R. 
\]
\begin{example} \cite[Lemma 5.5]{CY} 
\label{ex:Airy2}
For all $\ell, m \in \bbn$, ${\mathfrak{F}}_{x, \sym}^{2\ell,2m}(\Psi_{\Ai})$ has a basis given by
\[
\{ L(x, \partial_x)^j x^k + x^k L(x, \partial_x)^j : 0\leq j\leq l,\ 0\leq k\leq m\},
\]
and in particular,
\[
{\mathfrak{F}}_{x, \sym}^{2\ell,2m}(\Psi_{\Ai}) = (\ell + 1)(m + 1).
\]
\end{example}

For $\epsilon>0$ consider the sector
\[
\Sigma_\epsilon = \{re^{i\theta}\in\bbc: r> 0,\ |\theta| < \pi/6-\epsilon\}.
\]
The Airy function $\Ai(x)$ of the first kind
is holomorphic on this domain and has the asymptotic expansion
\[
\Ai(x) = e^{-\frac{2}{3}x^{3/2}}\Bigg(\sum_{j=1}^\infty c_j x^{-j/4}\Bigg)
\]
for some $c_j\in\bbr$ where $x^{1/4}$ is interpreted as the principal $4$th root of $x$.
Furthermore, any rational Darboux transformation of $\Psi_{\Ai}(x,z)$ equals
$\Psi(x,z) = \frac{1}{p(x)q(z)} P(x, \partial_x)\Psi_{\Ai}(x,z)$ for some polynomials $p(x),q(z)$ and a differential operator $P(x, \partial_x)$ with polynomial coefficients.
Thus, for any bispectral Darboux transformation of $\Psi_{\Ai}(x,z)$ we have the asymptotic estimate
\[
\|\partial_x^j\partial_z^k \Psi(x,z)\| = e^{-\frac{2}{3}(x+z)^{3/2}}\mathcal O((|x|+|z|)^{(j+k)/2+m})
\]
on $\Sigma_\epsilon$ for some integer $m$. 
The transformation $z\mapsto (2/3)z^{3/2}$ sends $\Sigma_\epsilon$ to the sector $\{re^{i\theta}\in\bbc: r>0,\ |\theta| < \pi/4-3\epsilon/2\}$.
Therefore if $\Gamma_1,\Gamma_2\subseteq\Sigma_\epsilon$ are smooth, semi-infinite curves inside this domain with parametrizations 
$\gamma_i(t): [0,\infty)\rightarrow\bbc$ then the real part of $-2(\gamma_1(t)+\gamma_2(s))^{3/2}/3$ will go to $-\infty$ as $t\rightarrow\infty$ or $s\rightarrow\infty$.
The above asymptotic estimate now shows that $\Psi(x,z)$ satisfies
\[
\int_{\Gamma_1} |x^mz^n\partial_x^j\partial_z^k  \Psi(x,z)| dx\in L^\infty(\Gamma_2) 
\quad
\mbox{and} 
\quad
\int_{\Gamma_2} |x^mz^n\partial_x^j\partial_z^k  \Psi(x,z)| dz\in L^\infty(\Gamma_1),
\]
for every pair of smooth, semi-infinite curves $\Gamma_1,\Gamma_2\subseteq\Sigma_\epsilon$.

Recall that the {\em{bilinear concomitant}} of a differential operator
\[
R(x, \partial_x) = \sum_{j=0}^m d_j(x)\partial_x^j.
\]
is the bilinear form $\mathcal{C}_{R}(-,-; p)$ defined on pairs of functions $f(x),g(x)$, which are analytic at $p \in \bbc$  by
\begin{align*}
\mathcal{C}_{R}(f,g;p)
  & = \sum_{j=1}^m \sum_{k=0}^{j-1} (-1)^k f^{(j-1-k)}(x)(d_j(x)g(x))^{(k)}|_{x=p}\\
  & = \sum_{j=1}^m \sum_{k=0}^{j-1}\sum_{\ell=0}^k\binom{k}{\ell} (-1)^k f^{(j-1-k)}(x)d_j(x)^{(k-\ell)}g(x)^{(\ell)}|_{x=p}.
\end{align*}
See for example \cite[Chapter 5, Section 3]{Ince}.

\begin{theorem} \cite{CY} Let  $\Psi(x,z)$ be a self-adjoint bispectral Darboux transformation of the Airy bispectral function $\Psi_{\Ai}(x,z)$
of bidegree $(d_1, d_2)$ and let $\Gamma_1$ 
and $\Gamma_2$ be two semi-infinite, smooth curves in $\Sigma_\epsilon$ for some $\epsilon > 0$, whose finite endpoints are $t_1$ and $t_2$, respectively.
Assume moreover that $\Psi(x,z)$ is holomorphic in a neighborhood of $\Gamma_1\times\Gamma_2$ and that the operators in ${\mathcal{F}}_x(\Psi)$ and ${\mathcal{F}}_z(\Psi)$ 
have holomorphic coefficients in a neighborhood of $\Gamma_1$ and $\Gamma_2$, respectively. Then the following hold:
\begin{enumerate}
\item $\dim {\mathfrak{F}}_{x, \sym}^{2\ell,2m}(\Psi)\geq (\ell + 1)(m + 1) + 1 - d_1d_2$.
\item If a differential operator $S(z, \partial_z) \in {\mathfrak{F}}_{z, \sym}(\Psi)$ satisfies
\[
\mathcal{C}_{S}(-,-;t_1) \equiv 0 \quad \mbox{and} \quad \mathcal{C}_{b_{\Psi}^{-1}(S) }(-,-;t_2) \equiv 0,
\]
then it commutes with the integral operator
\[
T: f(z)\mapsto \int_{\Gamma_1} K(z,w)f(w)dw \quad \text{with kernel} \quad K(z,w) = \int_{\Gamma_2} \Psi(x,z) \Psi(x,w) dx.
\]
\item If $\dim\mathfrak{F}_{z,\sym}^{2\ell,2\ell} \geq \ell(\ell+1) + 2$, in particular if $\ell=d_1d_2$, then there exists a differential operator $S(z, \partial_z) \in \mathfrak{F}_{z,\sym}^{\ell,\ell}(\Psi)$
of positive order satisfying the assmption and conclusion in part (2). 
\end{enumerate}
\end{theorem}
As a special case of this theorem, we are able to recover the commuting integral and differential operators studied by Tracy and Widom in \cite{TW2}.
In particular, if we take $\Psi =\Psi_{\Ai}$ in the theorem, then it guarantees the existence of a differential operator of order $2$ commuting with the integral operator
$$T_{\Ai} f(z)\mapsto \int_{t_1}^\infty K_{\Ai}(z,w)f(w)dw,$$
with kernel
$$K_{\Ai}(z,w) = \int_{t_2}^\infty \Ai(x+z)\Ai(x+w)dx = \frac{\Ai'(t_2+z)\Ai(t_2+w)-\Ai(t_2+z)\Ai'(t_2+w)}{z-w}.$$
Solving the associated system of linear equations for the vanishing concomitant, we discover that the differential operator 
$$S_{\Ai}(z,\partial_z) := \partial_z(t_1-z)\partial_z + (t_2-t_1)z + z^2$$
satisfies the condition that $\mathcal{C}_{S_{\Ai}}(f,g;t_1) = 0$ for all functions $f,g$ analytic at $t_1$.
Its preimage under the generalized Fourier map
$$b_{\Psi}^{-1}(S_{\Ai}(z,\partial_z)) = \partial_x(t_2-x)\partial_x + (t_1-t_2)x + x^2$$
also satisfies the condition $\mathcal{C}_{b_{\Psi}^{-1}S_{\Ai}}(f,g;t_2)  = 0$ for all functions $f,g$ analytic at $t_2$.
Therefore the differential operator $S_{\Ai}(z,\partial_z)$ commutes with $T_{\Ai}$.
This is precisely the differential operator discovered by Tracy and Widom in \cite{TW2}.

\section{Classification of self-adjoint rational Darboux transformations of the bispectral Airy function}
In this section, we will classify the self-adjoint rational Darboux transformations of the bispectral Airy function by leveraging two tools: (1) the technology of differential Galois theory, and (2) the classification of self-adjoint Darboux transformations in terms of Lagrangian subspaces of symplectic vector spaces found in \cite{CY}.
A similar classification is performed in \cite{BHY2} using the entirely different technique of performing an explicit asymptotic analysis of Wronskians associated to subspaces of the kernel.
More explicity, in this section we wish to classify factorizations of the form
\begin{equation}\label{eqn:main factorization}
P(x,\partial_x)^*\frac{1}{p(x)^2}P(x,\partial_x) = q(L(x,\partial_x))^2
\end{equation}
where here $p$ and $q$ are polynomials and $P(x,\partial_x)$ is a differential operator with polynomial coefficients.
Without loss of generality, we take $q(z)$ to be monic so that $p(x)$ is the leading coefficient of the operator $P(x,\partial_x)$.
The associated self-adjoint rational Darboux transformation of the bispectral function $\Ai(x+z)$ is then defined by
$$\Psi(x,z) = \frac{1}{p(x)q(z)}P(x,\partial_x)\cdot \Ai(x+z).$$

\subsection{Lagrangian Subspaces and Concomitant}
We begin by recalling the classification of self-adjoint factorizations of self-adjoint differential operators found in \cite{CY}.
To begin, let $A(x,\partial_x)$ be a differential operator and recall the standard fact that the concomitant $\mathcal{C}_A(f,g;x)$ is independent of $x$ for all $f\in\ker(A)$ and $g\in\ker(A^*)$.
\begin{lemma}[\cite{Wilson2}, Section 3]
Let $A(x,\partial_x)$ be a linear differential operator.  Then the concomitant of $A$ defines a canonical nondegenerate pairing
$$\ker(A)\times\ker(A^*)\rightarrow\mathbb C,\ (f,g)\mapsto \mathcal{C}_A(f,g).$$
\end{lemma}
Combining this with the identity $\mathcal{C}_A(f,g) = -\mathcal{C}_{A^*}(g,f)$, we see that the concomitant restricts to a symplectic bilinear form on $\ker(A)$ when $A(x,\partial_x)$ is formally symmetric.

We will also rely on the following formula for concomitants of differential operator products.
\begin{lemma}[\cite{Wilson2}, Lemma 3.6]\label{lem:product concomitant}
Let $ A(x,\partial_x) = A_1(x,\partial_x)A_2(x,\partial_x)$.  Then
$$\mathcal{C}_A(f,g;x) = \mathcal{C}_{A_1}(A_2f,g;x) + \mathcal{C}_{A_2}(f,A_1^*g;x).$$
\end{lemma}
From this, we see that if $A=A^*$, then $\ker(A_2)\subseteq \ker(A)$ and $\ker(A_1^*)\subseteq \ker(A)$ 
are orthogonal under the pairing defined by the concomitant of $A(x,\partial_x)$.

As is well-known in the theory of factorizations of linear differential operators, a factorization of a differential operator 
$$A(x,\partial_x) = A_1(x,\partial_x)A_2(x,\partial_x)$$
corresponds to a choice of a subspace $V\subseteq \ker(A)$.  The subspace $V$ corresponds to the kernel of $A_2(x,\partial_x)$ and determines the value of the operator $A_2(x,\partial_x)$ up to a left multiple by a function of $x$.
As is readily seen from the previous lemma, the kernel of $A_1(x,\partial_x)^*$ is completely determined by $V$ and given by the orthogonal complement
$$V^\perp = \{g\in \ker(A^*): \mathcal{C}_A(f,g)=0\ \forall f\in V\}.$$
Thus to obtain factorizations of the form \eqref{eqn:main factorization}, we search in particular for subspaces $V\subseteq \ker(q(L)^2)$ satisfying $V^\perp = V$.  In other words, we search for Lagrangian subspaces of the symplectic vector space $\ker(q(L)^2)$.  To summarize, we have the following proposition.
\begin{proposition}
Factorizations of the form \eqref{eqn:main factorization} with $p(x)$ and the coefficients of $P(x,\partial_x)$ not necessarily rational functions, correspond precisely to Lagrangian subspaces of the symplectic vector space $\ker(q(L)^2)$ whose symplectic form is defined by the concomitant of $q(L)^2$.
\end{proposition}

\subsection{Differential Galois Theory}
Our next task is to determine the symmetric factorizations obtained in the previous section which are rational.
For the convenience of the reader, we briefly outline the requisite basics of Picard-Vessiot extensions and the Fudamental Theorem of Differential Galois Theory.  We direct the interested reader to \cite{Put} for a more thorough treatment.

\begin{definition}
Let $(K,\partial)$ be a differential field and let $A\in K[\partial]$ be a linear differential operator with coefficients in $K$.
The \vocab{Picard-Vessiot} extension of $K$ associated with $A(x,\partial_x)$ is a differential field extension $(F,\partial)$ of $K$ whose constants all belong to $K$ and which is generated by the solutions of the homogeneous equation $Ag = 0$.
\end{definition}

Picard-Vessiot extensions of a differential field play precisely the role of Galois extensions in field theory.  Likewise, the usual Galois group is replaced by a similar object consisting of field automorphisms respecting differentiation.
\begin{definition}
The \vocab{differential Galois group} $\DGal(F/K)$ consists of all $K$-linear field automorphisms 
$\sigma: F \rightarrow F$ of $F$ satisfying $\sigma(\partial\cdot a) = \partial\cdot\sigma(a)$ for all $a\in F$.
\end{definition}

Analogous to the case of Galois extensions of fields, we have the following theorem relating differential subextensions and Zariski-closed subgroups of the differential Galois group (see \cite[Proposition 1.34]{Put}).
\begin{theorem}[Fundamental Theorem of Differential Galois Theory]
Let $(K,\partial)$ be a differential field whose subfield of constants is algebraically closed and let $(F,\partial)$ be a Picard-Vessiot extension of $K$.
Then there is a bijective correspondence between differential subfields $K\subseteq F'\subseteq F$ and Zariski-closed subgroups $G'\subseteq \DGal(F/K)$ given by
$$G'\subseteq \DGal(K/F) \mapsto K^{G'} = \{a\in K: \sigma(a) = a,\ \forall \sigma \in G'\}.$$
$$K\subseteq F'\subseteq F\mapsto \DGal(F'/K) = \{\sigma\in \DGal(F/K): \sigma(a)=a,\ \forall a\in F'\}.$$
Furthermore, this correspondence restricts to a correspondence between Picard-Vessiot subextensions of $F/K$ and normal subgroups of $\DGal(F/K)$.
\end{theorem}

We will not rely on the full force of this correspondence, 
and therefore will not have to recall the precise nature of the topological structure of $\DGal(F/K)$ as a group subscheme of a general linear group.
Instead, we will use only the immediate fact that
\begin{equation}
K = \{a\in F: \sigma(a) = a,\ \forall \sigma\in \DGal(F/K)\}.
\end{equation}

Since differential operators are determined (up to a multiple) by their kernels, rationality of a differential operator may be characterized by differential Galois invariance of the associated kernel.
\begin{theorem}
Let $A(x,\partial_x)$ be a differential operator with rational coefficients and let $F$ be the Picard-Vessiot extension of $\bbc(x)$ for $A$.
Consider a factorization $A(x,\partial_x) = A_1(x,\partial_x)A_2(x,\partial_x)$ with $A_2$ monic.  
Then $A_1(x,\partial_x)$ and $A_2(x,\partial_x)$ have rational coefficients if and only if $\ker(A_2)\subseteq\ker(A)$ is invariant under the action of $\DGal(F/\mathbb C(x))$.
\end{theorem}
\begin{proof}
For $\sigma\in \DGal(F/\mathbb C(x))$, let $\sigma(A_j) := \sigma(A_j)(x,\partial_x)$ denote the operator obtained by applying the automorphism to the coefficients.
Since the automorphism preserves differentiation, we know that
$$\sigma(A_j)(x,\partial_x)\cdot \sigma(a) = \sigma(A_j(x,\partial_x)\cdot a),\ \forall a\in F.$$
If $A_1(x,\partial_x)$ and $A_2(x,\partial_x)$ have rational coefficients, then clearly $\sigma(A_j) = A_j$ and therefore $\ker(\sigma(A_j)) = \ker(A_j)$.
Thus the kernel of $A_j(x,\partial_x)$ is invariant under the action of $\DGal(F/\mathbb C(x))$.

Conversely, suppose that $\ker(A_2)\subseteq \ker(A)$ is invariant under the action of the differential Galois group, ie. $\sigma(\ker(A_2)) = \ker(A_2)$
Then $\sigma(A_2)\cdot \sigma(a) = \sigma(A_2\cdot a) = \sigma(0) = 0$ for all $a\in \ker(A)$ and therefore $\ker(A_2) \subseteq \ker(\sigma(A_2))$.  Since the order of $A_2$ and the order of $\sigma(A_2)$ are the same, their kernels will have the same dimension.  Therefore $\ker(\sigma(A_2)) = \ker(A_2)$ and consequently $\sigma(A_2) = bA_2$ for some $b\in F$.  Since $A_2$ has leading coefficient $1$, it follows that $b=1$.  Hence $\sigma(A_2)=A_2$ and from the Fundamental Theorem of Differential Galois Theory, the coefficients of $A_2$ must all be rational functions.  Lastly, since $A$ and $A_2$ have rational coefficients, it follows that $A_1$ has rational coefficients.
\end{proof}

\begin{corollary}
Let $A(x,\partial_x)$ be a self-adjoint differential operator with rational coefficients and let $F$ be the Picard-Vessiot extension for $A$.
Then the self-adjoint, rational factorization of $A(x,\partial_x)$ correspond precisely with the $\DGal(F/\mathbb C(x))$-invariant Lagrangian subspaces of $\ker(A)$.
\end{corollary}
\begin{proof}
This follows immediately from the theorem and the results of the previous subsection.
\end{proof}

\subsection{The classification}
Now let $a_1,\dots, a_r\in \mathbb C$ be the distinct roots of $q(z)$ and write
$$q(z) = (z-a_1)^{d_1}\dots(z-a_r)^{d_r}$$
for some positive integers $d_1,\dots, d_r$ and distinct $a_1, \ldots, a_r \in \bbc$.
The kernel of $q(L)^2$ for $L(x,\partial_x) = \partial_x^2-x$ the Airy operator is given by the following lemma.
\begin{lemma}\label{lem:kernel basis}
The kernel of $q(L)^2$ has basis given by 
$$\{ \Ai^{(j)}(x+a_i),\Bi^{(j)}(x+a_i): 1\leq k\leq r,\ 0\leq j\leq 2d_k\}$$
where here $\Ai(x)$ and $\Bi(x)$ are the Airy functions of the first and second kind, respectively.
\end{lemma}
\begin{proof}
To prove this, we will rely on the fundamental relation
$$L(x,\partial_x)\partial_x = \partial_x L(x,\partial_x) + 1,$$
which implies that
$$(L(x,\partial_x)-a_k)^m\partial_x^n = \sum_{j=0}^{m\wedge n} \binom{m}{j}\frac{n!}{(n-j)!}\partial_x^{n-j}(L(x,\partial_x)-a_k)^{m-j}.$$
Thus for all $0\leq n < 2d_k-1$, we have
$$(L(x,\partial_x)-a_k)^{2d_k}\Ai^{(n)}(x+a_k) = \sum_{j=0}^{n} \binom{2d_k}{j}\frac{n!}{(n-j)!}\partial_x^{n-j}(L(x,\partial_x)-a_k)^{2d_k-j}\Ai(x+a_k) = 0.$$
Hence $\Ai^{(n)}(x+a_k)\in \ker((L(x,\partial_x)-a_k)^{2d_k})\subseteq \ker(q(L)^2)$ for all $0\leq n < 2d_k$.  The same calculation shows that $\Bi^{(n)}(x+a_k)\in \ker(q(L)^2)$ for all $0\leq n < 2d_k$.
\end{proof}

Thus the Picard-Vessiot extension of the differential field $(\mathbb C(x),\partial_x)$ corresponding to the linear differential operator $q(L(x,\partial_x))^2$ is finitely generated by $4r$ elements
$$F^q = \mathbb C(x)(\Ai(x+a_k),\Ai'(x+a_k),\Bi(x+a_k),\Bi'(x+a_k): 1\leq k \leq r).$$
Using this, we see that the differential Galois group of $F$ is isomorphic to $r$ copies of $\text{SL}_2(\mathbb C)$.
\begin{lemma}
The differential Galois group consists of all differential $\mathbb C(x)$-linear morphisms
$$\sigma: F^q\rightarrow F^q,\quad \left\lbrace\begin{array}{cc} \Ai(x+a_k)\mapsto \alpha_k\Ai(x+a_k) + \beta_k\Bi(x+a_k)\\\Bi(x+a_k)\mapsto \gamma_k\Ai(x+a_k) + \delta_k\Bi(x+a_k)\end{array}\right.\ \forall 1\leq k\leq r,$$
where here $\alpha_k,\beta_k,\gamma_k,\delta_k\in \mathbb C$ with $\alpha_k\delta_k-\beta_k\gamma_k = 1$.
\end{lemma}
\begin{proof}
The fact that
$$\begin{array}{c}\Ai(x+a_k)\mapsto \alpha\Ai(x+a_k) + \beta\Bi(x+a_k)\\ \Bi(x+a_k)\mapsto \gamma\Ai(x+a_k) + \delta\Bi(x+a_k)\end{array},\quad\left[\begin{array}{cc}\alpha & \beta\\\gamma & \delta\end{array}\right]\in\text{SL}_2(\mathbb C),$$
is a differential automorphism is standard.  See for example \cite[Example 8.15]{Put}.  Therefore, we need only show that this accounts for all differential automorphisms.

If $\sigma: F^q\rightarrow F^q$ is a differential automorphism fixing $\mathbb C(x)$, then
$$\frac{\sigma(\Ai(x+a_k))''}{\sigma(\Ai(x+a_k))} = \sigma\left(\frac{\Ai''(x+a_k)}{\Ai(x+a_k)}\right) = \sigma(x+a_k) = x+ a_k.$$
Thus $\sigma(\Ai(x+a_k))$ must be a solution of the differential equation $y'' = (x+a_k)y$, and therefore a linear combination of $\Ai(x+a_k)$ and $\Bi(x+a_k)$ for all $k$.
A similar statement holds for $\sigma(\Bi(x+a_k))$ so that 
$$\sigma: \left\lbrace\begin{array}{cc} \Ai(x+a_k)\mapsto \alpha_k\Ai(x+a_k) + \beta_k\Bi(x+a_k)\\\Bi(x+a_k)\mapsto \gamma_k\Ai(x+a_k) + \delta_k\Bi(x+a_k)\end{array}\right.\ \forall 1\leq k\leq r$$
for some $\alpha_k,\beta_k,\gamma_k,\delta_k\in \mathbb C$.
Lastly, the Wronskian identity implies
$$W(\Ai(x+a_k),\Bi(x+a_k)) = \Ai'(x+a_k)\Bi(x+a_k)-\Ai(x+a_k)\Bi'(x+a_k) = \frac{1}{\pi} \cdot$$
Since the Wronskian is skew-symmetric, we can conclude that
\begin{align*}
\frac{1}{\pi}
  &= \sigma(W(\Ai(x+a_k),\Bi(x+a_k)))\\
  &= W(\sigma(\Ai(x+a_k)),\sigma(\Bi(x+a_k)))\\
  &= (\alpha\delta-\beta\gamma)W(\Ai(x+a_k),\Bi(x+a_k)) = (\alpha\delta-\beta\gamma)/\pi.
\end{align*}
Hence $\alpha\delta-\beta\gamma = 1$.
\end{proof}

Using this, we can obtain the following characterization of the Galois-invariant subspaces of $\ker(q(L)^2)$.
\begin{lemma}\label{lem:invariant subspaces}
Suppose that $V\subseteq\ker(q(L)^2)$ is a subspace.  Then $V$ is invariant under the action of the differential Galois group if and only if $V$ is spanned by \emph{pairs} of elements of the form
$$\sum_{j=0}^{2d_k-1}\alpha_{kj}\Ai^{(j)}(x+a_k),\quad \sum_{j=0}^{2d_k-1}\alpha_{kj}\Bi^{(j)}(x+a_k)$$
\end{lemma}
\begin{proof}
Clearly any subspace spanned by pairs of elements of this form is invariant under the action of the Galois group, since the the action restricts to an action sending each of the functions in the pair to a linear combination of the functions in the pair.  Thus it suffices to show the converse.

Let $f(x)$ be a nonzero element of $V$.  Then
$$f(x) = \sum_{k=1}^r\sum_{j=0}^{2d_k-1} \alpha_{kj}\Ai^{(j)}(x+a_k) + \beta_{kj}\Bi^{(j)}(x+a_k)$$
Consider the differential automorphisms $\sigma_k$ and $\tau_k$ which fix $\Ai(x+a_j)$ and $\Bi(x+a_j)$ and satisfy
$$\sigma_k: \Ai(x+a_k)\mapsto -\Bi(x+a_k),\ \Bi(x+a_k)\mapsto \Ai(x+a_k),$$
$$\tau_k: \Ai(x+a_k)\mapsto \Ai(x+a_k) + \Bi(x+a_k),\ \Bi(x+a_k)\mapsto \Bi(x+a_k).$$
We see that
$$\tau_k(f(x))-f(x) = \sum_{j=0}^{2d_k-1}\alpha_{kj}\Bi^{(j)}(x+a_k)\in V.$$
Following up by applying $\sigma_k$, we see that 
$$\sigma_k(\tau_k(f(x))-f(x)) = \sum_{j=0}^{2d_k-1}\alpha_{kj}\Ai^{(j)}(x+a_k)\in V.$$
Likewise, one may show $\sum_{j=0}^{2d_k-1}\beta_{kj}\Ai^{(j)}(x+a_k),\ \sum_{j=0}^{2d_k-1}\beta_{kj}\Bi^{(j)}(x+a_k)\in V$ and since $k$ was arbitrary, the statement of the Lemma follows immediately.
\end{proof}

Our explicit description of the kernel of $q(L)^2$ allows us to give a concrete formula for the symplectic form on $\ker(q(L)^2)$ defined by the bilinear concomitant.
We start with a combinatorial Lemma.
\begin{lemma}
Let $a,b,m$ be integers.  Then
$$\sum_{k=0}^m(-1)^k\binom{k+a}{k}\binom{b}{m-k} = \binom{b-1-a}{m}.$$
\end{lemma}
\begin{proof}
We use the binomial series expansion on the identity
$$(1-z)^{-a-1}(1-z)^{b} = (1-z)^{b-a-1}$$
to find
$$\sum_{j,k=0}^\infty \binom{-a-1}{k}\binom{b}{j}(-1)^{j+k}z^{j+k} = \sum_{m=0}^\infty \binom{b-1-a}{m}(-1)^mz^m.$$
Now comparing coefficients of $z^m$:
$$\sum_{k=0}^m\binom{-a-1}{k}\binom{b}{m-k}(-1)^{j+k} =  \binom{b-1-a}{m}.$$
Noting that $\binom{-a-1}{k} = (-1)^k\binom{k+a}{a}$, the statement of the lemma follows immediately.
\end{proof}

\begin{proposition}
Let $f(x),g(x)\in\{\Ai(x),\Bi(x)\}$ and choose $0\leq m < 2d_j$ and $0\leq n < 2d_k$.  Then
$$\mathcal{C}_{q(L)^2}(f^{(m)}(x+a_j),g^{(n)}(x+a_k)) = \delta_{jk}\frac{m!n!W(f,g)}{(m+n-2d_k+1)!}\left.\partial_z^{m+n-2d_k+1}\right\rvert_{z=a_k}\cdot\left(\frac{q(z)^2}{(z-a_k)^{2d_k}}\right)$$
for all nonnegative integers $m,n$ with $m+n\geq 2d_k-1$ and is zero otherwise.
\end{proposition}
\begin{proof}
For simplicity of notation, we will let $h(z) = q(z)^2$ and write $f$ and $g$ in place of $f(x+a_k)$ and $g(x+b_k)$, respectively.
First note that if $j\neq k$ then $f^{(m)}\in \ker((L-a_j)^{2d_j})$ and $g^{(n)}\in\ker(h(L)(L-a_j)^{-2d_j})$, which is the orthogonal complement of the subspace $\ker((L-a_j)^{2d_j})$ of $\ker(h(L))$.  Hence $\mathcal{C}_{h(L)}(f^{(m)},g^{(n)}) = 0$.
Thus it suffices to consider the case when $j=k$.

Let $\wt h(z) = h(z)/(z-a_k)^{2d_k}$.
Applying Lemma \ref{lem:product concomitant} and the fundamental relation $L\partial_x = \partial_xL + 1$ we see that
\begin{align*}
\mathcal{C}_{h(L)}(f^{(m)},g^{(n)})
  & = \mathcal{C}_{h(L)\partial_x^m}(f,g^{(n)})\\
  & = \mathcal{C}_{\wt h(L)(L-a_k)^{2d_k}\partial_x^m}(f,g^{(n)})\\
  & = \sum_{s=0}^m \binom{m}{s}\frac{(2d_k)!}{(2d_k-s)!} \mathcal{C}_{\wt h(L)\partial_x^{m-s}(L-a_k)^{2d_k-s}}(f,g^{(n)}).
\end{align*}
Now using the fact that the concomitant of $L$ is the Wronskian and again applying Lemma \ref{lem:product concomitant} and the more general relation
$$\wt h(L(x,\partial_x))\partial_x^m = \sum_{s=0}^m\binom{m}{s}\partial_x^{m-s}\wt h^{(s)}(L(x,\partial_x))$$
we see that
\begin{align*}
\mathcal{C}_{h(L)}&(f^{(m)},g^{(n)})\\
  & = \sum_{s=0}^m \binom{m}{s}\frac{(2d_k)!}{(2d_k-s)!}(-1)^{m-s} W(f,(L-a_k)^{2d_k-s-1}\partial_x^{m-s}\wt h(L)\partial_x^n\cdot g)\\
  & = \sum_{t=0}^n\binom{n}{t}\wt h^{(t)}(a_k)\sum_{s=0}^m \binom{m}{s}\frac{(2d_k)!}{(2d_k-s)!}(-1)^{m-s} W(f,(L-a_k)^{2d_k-s-1}\partial_x^{n+m-s-t}\cdot g).
\end{align*}
From this it is clear that if $n+m<2d_k-1$ then the concomitant is zero.  Thus without loss of generality we take $m+n\geq 2d_k-1$.  Then for $\ell = n+m-2d_k+1$
\begin{align*}
\mathcal{C}_{h(L)}&(f^{(m)},g^{(n)})\\
  & = \sum_{t=0}^n\binom{n}{t}\wt h^{(t)}(a_k)\sum_{s=0}^m \binom{m}{s}\frac{(2d_k)!}{(2d_k-s)!}(-1)^{m-s} W(f,(L-a_k)^{2d_k-s-1}\partial_x^{n+m-s-t}\cdot g)\\
  & = \sum_{t=0}^{n\wedge\ell}\binom{n}{t}\wt h^{(t)}(a_k)\sum_{s=0}^m \binom{m}{s}\frac{(2d_k)!}{(2d_k-s)!}(-1)^{m-s}\frac{(m+n-s-t)!}{(\ell-t)!} W(f,\partial_x^{\ell-t}\cdot g)\\
  & = \sum_{t=0}^{n\wedge\ell}\binom{n}{t}\wt h^{(t)}(a_k)\frac{m!(n-t)!}{(\ell-t)!}\sum_{s=0}^m \binom{2d_k}{s}\binom{m+n-s-t}{m-s}(-1)^{m-s} W(f,\partial_x^{\ell-t}\cdot g).
\end{align*}
Now reindexing the sum and applying the previous lemma, we obtain
\begin{align*}
\mathcal{C}_{h(L)}&(f^{(m)},g^{(n)})\\
  & = \sum_{t=0}^{n\wedge\ell}\binom{n}{t}\wt h^{(t)}(a_k)\frac{m!(n-t)!}{(\ell-t)!}\sum_{s=0}^m \binom{2d_k}{m-s}\binom{s+n-t}{s}(-1)^{s} W(f,\partial_x^{\ell-t}\cdot g)\\
  & = \sum_{t=0}^{n\wedge\ell}\binom{n}{t}\wt h^{(t)}(a_k)\frac{m!(n-t)!}{(\ell-t)!}\binom{2d_k-1-n+t}{m} W(f,\partial_x^{\ell-t}\cdot g).
\end{align*}
The binomial coefficient in the last sum is nonzero if and only if $\ell\leq t$.  Since the sum is taken between $t=0$ and $t=\ell$, the only nonzero term comes from when $t=\ell$.  Thus
$$\mathcal{C}_{h(L)}(f^{(m)},g^{(n)}) = \frac{m!n!}{(n+m-2d_k+1)!}\wt h^{(n+m-2d_k+1)}(a_k)W(f,g).$$
\end{proof}

The rational Darboux transformations of $\Ai(x+z)$ come directly from factorizations of the form \eqref{eqn:main factorization} with $P(x,\partial_x)$ having rational coefficients.  As we have outlined above, these correspond precisely to the Galois-invariant Lagrangian subspaces of $\ker(q(L)^2)$.
This characterization is made explicit in the next theorem.

\begin{theorem}[Classification Theorem]
Let  $f_m,g_m\in \ker(q(L)^2)$ for $1\leq m\leq d$ be $2d$ linearly independent functions of the form
$$f_i(x) = \sum_{m=0}^{2d_{\ell_i}-1} \alpha_{im} \Ai^{(m)}(x + a_{\ell_i}),\quad g_i(x) = \sum_{n=0}^{2d_{\ell_i}-1} \alpha_{in} \Bi^{(n)}(x + a_{\ell_i})$$
satisfying the condition that
$$\sum_{m+n\geq 2d_k-1}^{2d_k-1}\alpha_{im}\alpha_{jn}\frac{m!n!}{(m+n-2d_k+1)!}\left.\partial_z^{m+n-2d_k+1}\right\rvert_{z=a_k}\cdot\left(\frac{q(z)^2}{(z-a_k)^{2d_k}}\right)=0$$
for all $k$ and for all $i,j$ with $\ell_i=\ell_j=k$.
Then the differential operator $P(x,\partial_x)$ of order $2d$ defined in terms of a Wronskian by
$$P(x,\partial_x)\cdot f := W(f_1,f_2,\dots,f_d,g_1,g_2,\dots, g_d,f)$$
has rational coefficients and satisfies
$$P(x,\partial_x)^*\frac{1}{p(x)^2}P(x,\partial_x) = q(L(x,\partial_x))^2$$
for some rational function $p(x)$.  Furthermore every self-adjoint rational factorization of $q(L(x,\partial_x))^2$ is of this form.
\end{theorem}
\begin{proof}
This follows directly from our direct calculation of the concomitant along with our characterization of the Galois-invariant subspaces of the kernel.
\end{proof}

This result is particularly nice in the situation that $q(z) = (z-s_1)^d$, so that the concomitant has the simple form
$$\mathcal{C}_{q(L)^2}(f^{(m)}(x+s_1),g^{(n)}(x+s_1)) =
\left\lbrace\begin{array}{cc}
m!n!/\pi,& f = \Ai,\ g = \Bi,\ m+n = 2d-1\\
-m!n!/\pi,& f = \Bi,\ g = \Ai,\ m+n = 2d-1\\
0,& \text{otherwise}.
\end{array}\right.$$
The payout of our dive through all the differential Galois theory and symplectic geometry above is that we immediately provide \emph{explicit factorizations} of $(L(x,\partial_x)-s_1)^2$ and $(L(x,\partial_x)-s_1)^4$.
\begin{corollary}\label{lem:rank 1 classification}
Let $s_1\in\mathbb C$.
Then up to a function multiple, the only self-adjoint rational factorizations of $(L(x,\partial_x)-s_1)^2$ are the trivial one and
$$P_1(x,\partial_x)^*\frac{1}{(x+s_1)^2}P_1(x,\partial_x) = (L(x,\partial_x)-s_1)^2$$
for
$$P_1(x,\partial_x) = (x+s_1)\partial_x^2 - \partial_x - (x+s_1)^2.$$
\end{corollary}
\begin{proof}
From the previous Theorem, we know we must choose functions
$$f_1(x) = \alpha_{11}\Ai(x+s_1) + \alpha_{12}\Ai'(x+s_1),\quad g_1(x) =  \alpha_{11}\Bi(x+s_1) + \alpha_{12}\Bi'(x+s_1)$$
satisfying $2\alpha_{11}\alpha_{12}1!1!/\pi = 0$.
Thus either $\alpha_{11}=0$ or $\alpha_{12}=0$ and without loss of generality we may take the remaining coefficient to be $\pi$.
In the first case, the operator $P(x,\partial_x)$ is
$$P(x,\partial_x)\cdot f = W(\Ai'(x+s_1),\Bi'(x+s_1),f) = (x+s_1)f''(x) - f'(x) - (x+s_1)^2f.$$
In the second case, the operator $P(x,\partial_x)$ is
$$P(x,\partial_x)\cdot f = W(\Ai(x+s_1),\Bi(x+s_1),f) = f''(x)-(x+s_1)f.$$
Thus in this second case $P(x,\partial_x) = L(x,\partial_x)-s_1$, giving us the trivial factorization of $(L(x,\partial_x)-s_1)^2$.
\end{proof}

\begin{corollary}\label{lem:rank 2 classification}
Let $s_1\in\mathbb C$.
Then up to a function multiple, the self-adjoint rational factorizations of $(L(x,\partial_x)-s_1)^4$ are of the form
$$P_2(x,\partial_x)^*\frac{1}{(x+s_1)^2}P_2(x,\partial_x) = (L(x,\partial_x)-s_1)^4$$
for
$$P_2(x,\partial_x)\cdot f = W(f_1,f_2,g_1,g_2,f),$$
where here
$$f_k(x) = \sum_{j=0}^3\alpha_{kj}\Ai^{(j)}(x+s_1),\quad g_k(x) =  \sum_{j=0}^3\alpha_{kj}\Bi^{(j)}(x+s_1)$$
for some constants $\alpha_{kj}$ satisfying the three relations
$$6\alpha_{m3}\alpha_{n0} + 2\alpha_{m1}\alpha_{n2} + 2\alpha_{m2}\alpha_{n1} + 6\alpha_{m3}\alpha_{n0} = 0,\ \ 1\leq m\leq n\leq 2.$$
\end{corollary}
\begin{proof}
This follows immediately from the Classification Theorem.
\end{proof}

The operator $P_2(x,\partial_x)$ in this latter situation is more complicated.
First of all, it features the factorizations from the previous corollary, as may be obtained from taking $\alpha_{13}=\alpha_{23}=0$.
Thus to get new factorizations, we can without loss of generality take $\alpha_{13}=\alpha_{23}=1$.
Then the three relations simplify to $\alpha_{m0} = -\alpha_{m1}\alpha_{m2}/3$ for $m=1,2$ plus a choice of either $\alpha_{11} = \alpha_{21}$ or $\alpha_{12} = \alpha_{22}$.
For sake of concreteness, we choose $\alpha_{11}=\alpha_{21}$ and take $\alpha_{22}=1$, $\alpha_{12}=0$, and $s_1=0$.
This determines all parameters, except for $\alpha_{11}$ and the associated operator $P(x,\partial_x)$ is explicitly computed to be
{\small\begin{align}\label{eqn:P2}
\nonumber P_2(x,\partial_x)
  & = \left(x^4 - 4x^3\alpha_{11} + \frac{10}{3}x^2\alpha_{11}^2 + \left(\frac{4}{3}\alpha_{11}^3 + 4\right)x + \frac{1}{9}\alpha_{11}^4-8\alpha_{11}\right)\partial_x^4\\\nonumber
  & + \left(-4x^3 + 12x^2\alpha_{11} - \frac{20}{3}\alpha_{11}^2x - \frac{4}{3}\alpha_{11}^3-4\right)\partial_x^3\\
  & + \left(-2x^5 + 8x^4\alpha_{11} - \frac{20}{3}\alpha_{11}^2x^3 - \left(\frac{8}{3}\alpha_{11}^3 + 2\right)x^2 - \left(\frac{2}{9}\alpha_{11}^4 - 4\alpha_{11}\right)x + \frac{10}{3}\alpha_{11}^2\right)\partial_x^2\\\nonumber
  & + \left(2x^4 - 4x^3\alpha_{11} - \frac{4}{3}x\alpha_{11}^3 - 16 - \frac{2}{9}\alpha_{11}^4 + 36\alpha_{11}\right)\partial_x\\\nonumber
  & + x^6 - 4x^5\alpha_{11} + \frac{10}{3}x^4\alpha_{11}^2 + \left(\frac{4}{3}\alpha_{11}^3 + 8\right)x^3 + \left(\frac{1}{9}\alpha_{11}^4-22\alpha_{11}\right)x^2+\frac{16}{3}x\alpha_{11}^2 + 2\alpha_{11}^3 + 16\nonumber
\end{align}
}

\section{Commuting differential operators for the level one kernels}

In this section, we explore the commuting differential operators for integral operators with \vocab{level one Airy kernels}, ie. those defined by bispectral functions $\Psi$ obtained from self-adjoint rational Darboux transformations of $(L(x,\partial_x)-s_1)^2$.
There is only one such bispectral function, determined by the factorization of $(L(x,\partial_x)-s_1)^2$ in Corollary \ref{lem:rank 1 classification}.  Using the operator $P_1(x,\partial_x)$ described in this Corollary, the associated bispectral function is
$$\Psi_1(x,z) = \frac{1}{(x+s_1)(z-s_1)}P_1(x,\partial_x)\cdot\Psi_{\Ai}(x,z) = \Ai(x+z) - \frac{1}{(x+s_1)(z-s_1)}\Ai'(x+z).$$
Let $\wt{P}_1(z,\partial_z) = b_{\Psi_{\Ai}}(P_1(x,\partial_x))$, $p_1(x) = x+s_1$ and $q_1(z)=z-s_1$.  For every $R(x,\partial_x)\in \mathfrak{F}_x(\Psi_{\Ai})$ and $S(z,\partial_z) = b_{\Psi_{\Ai}}(R(x,\partial_x))$, we have the identities
\begin{equation}
\frac{1}{p_1(x)}P_1(x,\partial_x)^*R(x,\partial_x)P_1(x,\partial_x)\frac{1}{p_1(x)}\cdot \Psi_1(x,z) = (q_1(z))S(z,\partial_z)(q_1(z))\cdot \Psi_1(x,z),
\end{equation}
\begin{equation}\frac{1}{q_1(z)}\wt{P}_1(z,\partial_z)^*S(z,\partial_z)\wt{P}_1(z,\partial_z)\frac{1}{q_1(z)}\cdot \Psi_1(x,z) = (p_1(x))R(x,\partial_x)(p_1(x))\cdot \Psi_1(x,z),
\end{equation}
and the more complicated identity
\begin{align}\label{eqn:complicated}
& \left(\frac{1}{p_1(x)}P_1(x,\partial_x)R(x,\partial_x)p(x) + p(x)R(x,\partial_x)^*P_1(x,\partial_x)^*\frac{1}{p_1(x)}\right)\cdot\Psi_1(x,z)\\\nonumber
& = \left(\frac{1}{q_1(z)}\wt{P}_1(z,\partial_z)S(z,\partial_z)q(z) + q(z)S(z,\partial_z)^*\wt{P}_1(z,\partial_z)^*\frac{1}{q_1(z)}\right)\cdot\Psi_1(x,z)
\end{align}
Comparing the orders of these operators, we see that $\mathfrak{F}_{x,\sym}^{2\ell,2m}(\Psi_1)$ contains the direct sum
\begin{align*}
\mathfrak{F}_{x,\sym}^{2\ell,2m}(\Psi_1)
  & \supseteq \frac{1}{p_1(x)}P_1(x,\partial_x)^*\mathfrak F_{x,\sym}^{2\ell-4,2m}(\Psi_{\Ai})P_1(x,\partial_x)\frac{1}{p_1(x)}\\
  & \oplus p_1(x)\mathfrak F_{x,\sym}^{2,2m-4}(\Psi_{\Ai}) p_1(x)\oplus \mathfrak E\oplus \mathbb C
\end{align*}
for all $\ell,m\geq 2$, where here $\mathfrak E$ is a set of additional operators stemming from the equation \eqref{eqn:complicated}
$$\mathfrak E = \left\lbrace\frac{1}{p_1(x)}P_1(x,\partial_x)R(x,\partial_x)p(x) + p(x)R(x,\partial_x)^*P_1(x,\partial_x)^*\frac{1}{p_1(x)}: R(x,\partial_x)\in\mathfrak F^{1,1}_x(\Psi_{\Ai})\right\rbrace.$$
Explicit calculation shows that $\mathfrak E$ is two dimensional.
Consequently the dimension of $\mathfrak{F}_{x,\sym}^{2\ell,2m}(\Psi_1)$ is at least $(\ell-1)(m+1) + 2(m-1) + 2 + 1 = (\ell+1)(m+1)-1$.
One can show that this is precisely the dimension for all $m,n>1$ and that both $\mathfrak F_x(\Psi)$ and $\mathfrak F_z(\psi)$ are equal to algebras of differential operators on a rational curve with a cusipdal singularity of degree $2$ at the origin.

Let $T_1$ be the integral operator
$$T_1: f(z)\mapsto \int_{t_1}^\infty K_1(z,w)f(w)dw,\quad K_1(z,w) = \int_{t_2}^\infty \Psi_1(x,z)\Psi_1(x,w)dx.$$
The specific value of the kernel $K_1(z,w)$ is determined via integration by parts to be
$$K_1(z,w) = \frac{q_1(w)}{q_1(z)}K_{\Ai}(z,w) + \mathcal{C}_{P_1}(\psi_{\Ai}(x,z),\psi_1(x,w)/p_1(x);t_2).$$
From the previous estimate of the dimension of $\mathfrak{F}_{x,\sym}^{2\ell,2m}(\Psi_1)$, we see that $T_1$ will commute with a differential operator $S_1(z,\partial_z)$ in $\mathfrak F_{z,\sym}^{4,4}(\Psi_1)$.

The values of the commuting integral and differential operators will in general depend on $s_1$, albeit predictably.
If we make the $s_1$-dependence of $\Psi(x,z)=\Psi(x,z;s_1)$ explicit, we see $\Psi(x,z;s_1)=\Psi(x+s_1,z-s_1;0)$ and consequently the differential operator $S_1(z,\partial_z)$ commuting with $T_1$ for arbitrary $s_1$ is the same as in the case $s_1=0$, but with $z$ replaced by $z-s_1$ and $t_2$ replaced by $t_2+s_1$.
Thus without loss of generality we will take $s_1=0$.

Explicitly computing the condition of the vanishing of the concomitant and solving the resulting linear system of equations yields the operator of order $4$
$$S_1(z,\partial_z) = \frac{1}{z}\left(\sum_{k=0}^2\partial_z^ka_k(z)(z-t_1)^k\partial_z^k\right)\frac{1}{z}$$
where here
\begin{align*}
a_2(z) &= z^2,\\
a_1(z) &= -2(z^4+(t_2-t_1)z^3 - 3t_1),\\
a_0(z) &= z^3(z^3 + 2(t_2-t_1)z^2 + (t_2-t_1)^2z - 8) + (t_1+t_2)z^2/3.
\end{align*}

The dimension estimates also imply the existence of a commuting differential operator of order $6$, which we find to be
$$\wt{S}_1(z,\partial_z) = \frac{1}{z}\left(\sum_{k=0}^3\partial_z^k\wt a_k(z)(z-t_1)^k\partial_z^k\right)\frac{1}{z}$$
where here
\begin{align*}
\wt a_3(z) =&\ z^2,\\
\wt a_2(z) =&\ -3(z^4+(t_2-t_1)z^3 - 4t_1),\\
\wt a_1(z) =&\ 3(z^6 + 2(t_2-t_1)z^5 + (t_2-t_1)^2z^4 - 10 z^3 + (5t_1-4t_2)z^2 - 3t_1(t_2-t_1)z),\\
\wt a_0(z) =&\ -z^8 - 3(t_2-t_1)z^7 -3(t_2-t_1)^2z^6 - ((t_2-t_1)^3 - 32)z^5\\
       &+ (42t_2-63t_1)z^4 + (36t_1^2-48t_1t_2+12t_2^2)z^3 + t_1t_2(t_1+t_2)z^2 +  12t_1^2 - 6t_1t_2.
\end{align*}
The operators $S_1(z,\partial_z)$ and ${\wt S}_1(z,\partial_z)$ commute and thus satisfy an algebraic relation.
The relation is
$$\wt{S}_1^2 = S_1^3 - \frac{t_1^2-t_1t_2+t_2^2}{3}S_1 + \frac{(t_1-2t_2)(2t_1-t_2)(t_1+t_2)}{3^3}.$$
The discriminant of the polynomial on the right hand side is
$$\Delta = -\frac{16}{27}(260t_1^6 - 780t_1^5t_2 - 627t_1^4t_2^2 + 2554t_1^3t_2^3 - 627t_1^2t_2^4 - 780t_1t_2^5 + 260t_2^6),$$
so for generic values of $t_1$ and $t_2$, the associated algebraic variety is an elliptic curve.

\section{Commuting differential operators for the level two kernels}

In this section, we explore the commuting differential operators for integral operators with \vocab{level two Airy kernels}, ie. those defined by bispectral functions $\Psi$ obtained from self-adjoint rational Darboux transformations of $(L(x,\partial_x)-s_1)^2(L(x,\partial_x)-s_2)^2$.  We will focus on the particular case when $s_1=s_2$, leaving the other situation to a future publication.
Note also that due to the nice translation behavior of $\psi_{\Ai}(x,z)$, we can easily rederive the formula for general values of $s_1$ from the case when $s_1=0$.  So for sake of simplicity, we will take $s_1=0$.

There are many bispectral functions in the level two case, all of which are determined by the factorizations of $L(x,\partial_x)^4$ in Corollary \ref{lem:rank 2 classification}, which in turn are determined by a choice of $\alpha_{jk}$ for $j=1,2$ and $0\leq k \leq 3$ satisfying the constraints of the Corollary.  The precise value $P_2(x,\partial_x)$ and the commuting operator is very complicated in general.  To facilitate our computations, and the inclusion of exact formulas in our paper, we will take $\alpha_{31}=\alpha_{32}=1$, $\alpha_{11}=\alpha_{21}$ and take $\alpha_{22}=1$, $\alpha_{12}=0$, so that $P_2(x,\partial_x)$ is given by \eqref{eqn:P2}.
Additionally we will take $\alpha_{11}=0$ so that $P_2(x,\partial_x)$ has the simplified formula
$$P_2(x,\partial_x) = x(x^3+4)\partial_x^4 - 4(x^3+1)\partial_x^3 - 2x^2(x^3+1)\partial_x^2+2x(x^3-8)\partial_x + x^6 + 8x^3 + 16.$$

Let $q_2(z)=z^2$ and $p_2(x) = x(x^3+4)$.
The corresponding bispectral function is defined by
\begin{align*}
\Psi_2(x,z)
  & = \frac{1}{p_2(x)q_2(z)}P_2(x,\partial_x)\cdot\Psi_{\Ai}(x,z)\\
  & = \Ai(x+z) + \frac{6(x^3 + x^2z + 2)}{p_2(x)q_2(z)}\Ai(x+z) - \frac{4(x^3w + 3x + w)}{p_2(x)q_2(z)}\Ai'(x+z).
\end{align*}

The Fourier algebras for $\Psi_2(x,z)$ are given by algebras of differential operators on some rank $1$, torsion-free modules over certain rational curves with cuspidal singularities.
Specifically, let $\mathscr A_x = \{f(x)\in\bbc[x]: p(x)|f'(x)\}$ be the coordinate ring of a singular rational curve $X$ with cusps of order $2$ at the roots of $p(x)$.  Then
$$\mathfrak F_x(\Psi_2) = \{D(x,\partial_x): D(x,\partial_x)\cdot\mathscr A_x\subseteq\mathscr A_x\}$$
is the algebra of differential operators on $X$.
Likewise, let $\mathscr A_z = \mathbb C[z^4,z^5]$ be the affine coordinate ring of a rational curve $Z$ with a higher-order cusp at $0$ and consider the torsion-free rank $1$ $\mathscr A_z$-module $\mathscr M_z = \vspan_{\mathbb C}\{z^{-2},z^{-1}\}\oplus z^2\mathbb C[z]$.
Then 
$$\mathfrak F_z(\Psi_2) = \{D(z,\partial_z): D(z,\partial_z)\cdot\mathscr M_z\subseteq\mathscr M_z\}$$
is the algebra of differential operators on the line bundle $\mathcal L$ over $Z$ associated to $\mathscr M_z$.

The generalized Fourier map $b_{\Psi}$ may be described in terms of $b_{\Psi_{\Ai}}$ by
$$b_{\Psi}(A(x,\partial_x)) = \frac{1}{q_2(z)}b_{\Psi_{\Ai}}\left[P_2(x,\partial_x)^*\frac{1}{p_2(x)}A(x,\partial_x)\frac{1}{p_2(x)}P_2(x,\partial_x)\right]\frac{1}{q_2(z)}.$$

Let $T_2$ be the integral operator
$$T_2: f(z)\mapsto \int_{t_1}^\infty K_2(z,w)f(w)dw,\quad K_2(z,w) = \int_{t_2}^\infty \Psi_2(x,z)\Psi_2(x,w)dx.$$
The specific value of the kernel $K_2(z,w)$ is determined via integration by parts to be
$$K_2(z,w) = \frac{q_2(w)}{q_2(z)}K_{\Ai}(z,w) + \mathcal{C}_{P_2}(\psi_{\Ai}(x,z),\psi_2(x,w)/p_2(x);t_2).$$
Computer calculation finds $\dim \mathfrak{F}_{x,\sym}^{10,10}(\Psi_2) = 32$, and therefore $T_2$ will commute with a differential operator $S_2(z,\partial_z)$ in $\mathfrak F_{z,\sym}^{10,10}(\Psi_2)$.

Taking $t_1=t_2=1$, and solving the linear system describing the vanishing of the concomitants, we find differential operators of order $10$, $12$, $14$, $16$, and $18$ commuting with $T_2$.
The operators $S_2(z,\partial_z)$ and $\wt S_2(z,\partial_z)$ of order $10$ and $12$ are given by
$$S_2(z,\partial_z) = \frac{1}{z^2}\left(\sum_{k=0}^{5}\partial_z^k(1-z)^ka_k(z)\partial_z^k\right)\frac{1}{z^2},$$
\begin{align*}
a_0(z) &= z^{14}-200z^{11} + 170z^{10} + 5640z^8 - 7360z^7 + 2160 z^6 - 11520 z^5 - 2880 z + 4320,\\
a_1(z) &= 5z^{12} - 580 z^9 + 380 z^8 + 6240 z^6 - 3700 z^5 - 960z^2 - 9600z +4800,\\
a_2(z) &= 10z^{10} - 560 z^7 + 180z^6 + 960 z^4 + 1800 z^3 + 300z^2,\\
a_3(z) &= 10z^8 - 180z^5 - 100z^4 - 420z + 1260,\\
a_4(z) &= 5z^6 - 70z^2,\\
a_5(z) &= z^4;
\end{align*}
$$\wt S_2(z,\partial_z) = \frac{1}{z^2}\left(\sum_{k=0}^{6}\partial_z^k(1-z)^k\wt a_k(z)\partial_z^k\right)\frac{1}{z^2},$$
\begin{align*}
\wt a_0(z) &= z^{16} - 340 z^{13}   + 504 z^{12}   + 21040 z^{10}   - 52200 z^9  + 28812 z^8\\
       &\ - 192000 z^7 + 490464 z^6  - 328320 z^5  - 201600 z + 130464,\\
\wt a_1(z) &= 6 (z^{14} - 220 z^{11}   + 300 z^{10}   + 7000 z^8  - 14212 z^7  + 5148 z^6 \\
       &\ - 16800 z^5  + 13568 z^4 + 13568 z^3  + 2368 z^2  - 6240 z + 12480,\\
\wt a_2(z) &= 3(5z^{12}   - 640 z^9  + 760 z^8  + 7800 z^6  - 8792 z^5  - 2996 z^4  - 3120 z^2  - 36000 z + 50400,\\
\wt a_3(z) &= 4z^2(5z^8 - 310z^5  + 270 z^4  + 600 z^2  + 1566 z - 2268),\\
\wt a_4(z) &= 3(5z^8-100z^5 -224z + 784),\\
\wt a_5(z) &= 6(z-2)z^2(z+2)(z^2+4),\\
\wt a_6(z) &= z^4.
\end{align*}

From Burchnall-Chaundy Theory and its extensions (see for example \cite{Krichever2}), we know that each pair of operators must satisfy a polynomial relation.  Together, the algebra they generate is the coordinate ring of an affine curve.
However, the precise relations that are satisfied are sufficiently complicated so as to be omitted from the paper.

{\bf Acknowledgements.}
The research of W.R.C. was supported by CSU Fullerton RSCA intramural grant 0359121.
The research of M.Y. was supported by NSF grant DMS-2131243.
The research of I.Z. was supported by CONICET PIP grant 112-200801-01533.

\bibliographystyle{plain}

\end{document}